\documentclass[10 pt,a4paper,twoside, reqno]{amsart}
\usepackage{amssymb,amsmath,amsfonts,amsthm,graphicx,amscd}
\usepackage{enumerate,mathrsfs,longtable,bm,float}
\usepackage[all,cmtip,line]{xy}
\usepackage{latexsym,epsfig} \usepackage{latexsym, url,amsxtra}
\usepackage{cite}
\usepackage{csquotes}

  \textwidth = 12 cm
 \DeclareMathAlphabet{\mathpzc}{OT1}{pzc}{m}{it}
\numberwithin{equation}{section}
\theoremstyle{definition}
\newtheorem*{theoA}{Theorem A}
\newtheorem*{theoB}{Theorem B}
\newtheorem*{theoC}{Theorem C}
\newtheorem*{theoD}{Theorem D}
\newtheorem*{theoE}{Theorem E}

\newtheorem{theo}{Theorem}[section]
\newtheorem{lem}{Lemma}[section]
\newtheorem{cor}{Corollary}[section]

\newtheorem{exm}{Example}[section]
\newtheorem{defi}{Definition}[section]
\newtheorem{rem}{Remark}[section]
\newtheorem{question}{Question}[section]

\newcommand{\ol}{\overline}
\newcommand{\be}{\begin{equation}}
\newcommand{\ee}{\end{equation}}
\newcommand{\beas}{\begin{eqnarray*}}
\newcommand{\eeas}{\end{eqnarray*}}
\newcommand{\bea}{\begin{eqnarray}}
\newcommand{\eea}{\end{eqnarray}}

\numberwithin{equation}{section}
\newcommand {\Jimsocity}{\hspace*{-0.6 cm}\noindent \textit{Journal of the Indian Math. Soc.\hfill{DOI. 10.18311/jims/2018/16578}\\
Vol. 85, Nos. (1 - 2),\,\,(2018),\,\,1--15.}\\}

\begin{document}

\thispagestyle{empty}
\Jimsocity
\begin{center}
    (Final version is available at http://www.informaticsjournals.com/index.php/jims)
\end{center}
\title
[On the uniqueness of power of a meromorphic function...]
{\scshape \Large \bf {On the uniqueness of power of a meromorphic function sharing a set with its $k-$th derivative}}

\author [A. Banerjee and B. Chakraborty]{\bf  Abhijit Banerjee  and Bikash Chakraborty }
\thanks{
\\ISSN 0019--5839  \hspace{0.3 cm}
\hfill{}\hspace{1 in}\hfill{ \textit{\large \copyright\,\,{Indian Mathematical Society, 2017}}}}
\begin{abstract}
In the existing literature, many researchers consider the uniqueness of the power of a meromorphic function with its derivative counterpart share certain values or small functions. Here we consider the same problem  under the aegis of a more general settings namely set sharing.
 \end{abstract}
 \maketitle

\section{Introduction and Definitions}
In 1925, R. Nevanlinna developed a systematic study of the value distribution theory by means of his First and Second Fundamental Theorems. In this paper, we assume that readers familiar with value distribution theory (\cite{3}).\par
By $\mathbb{C}$ and $\mathbb{N}$, we mean the set of complex numbers and the set of positive integers respectively. Let $f$ and $g$ be two non-constant meromorphic functions and let $a$ be a finite complex number. We say that $f$ and $g$ share the value $a$ CM (counting multiplicities), provided that $f-a$ and $g-a$ have the same zeros with the same multiplicities. Similarly, we say that $f$ and $g$ share the value $a$ IM (ignoring multiplicities), provided that $f-a$ and $g-a$ have the same set of zeros, where the multiplicities are not taken into account. In addition, we say that $f$ and $g$ share $\infty$ CM (resp. IM), if $1/f$ and $1/g$ share $0$ CM (resp. IM).\par
Now we recall the notion of weighted sharing which appeared in the literature in 2001 (\cite{4}) as this definition paves the way for future discussions as far as relaxation of sharing is concerned.
\begin{defi} (\cite{4}) Let $k$ be a non-negative integer or infinity. For $a\in\mathbb{C}\cup\{\infty\}$, we denote by $E_{k}(a;f)$ the set of all $a$-points of $f$, where an $a$-point of multiplicity $m$ is counted $m$ times if $m\leq k$ and $k+1$ times if $m>k$.\par
We say that $f$ and $g$ share the value $a$ with weight $k$ if $E_{k}(a;f)=E_{k}(a;g)$.
\end{defi}
We write $f$ and $g$ share $(a,k)$ to mean that $f$ and $g$ share the value $a$ with weight $k$. Clearly, if $f$ and $g$ share $(a,k)$, then also $f$ and $g$ share $(a,p)$ for any integer $p$ with $0\leq p<k$. Also we note that $f$ and $g$ share a value $a$ CM (resp. IM) if and only if $f$ and $g$ share $(a,\infty)$ (resp. $(a,0)$).\par
\begin{defi} Let $S$ be a set of distinct elements of $\mathbb{C}\cup\{\infty\}$ and $k$ be a non-negative integer or $\infty$. We denote by $E_{f}(S,k)$ the set $\bigcup\limits_{a\in S}E_{k}(a;f)$.\par
We say that $f$ and $g$ share the set $S$ with weight $k$ if $E_{f}(S,k)=E_{g}(S,k)$.
\end{defi}
\begin{defi} A set $S\subset \mathbb{C}\cup\{\infty\}$ is called a unique range set for meromorphic (resp. entire) functions with weight $k$, in short, $URSM_{k}$ (resp. $URSE_{k}$), if for any two non-constant meromorphic (resp. entire) functions $f$ and $g$ the condition $E_{f}(S,k)=E_{g}(S,k)$ implies $f\equiv g$.
\end{defi}
Next we recall following two definitions which have been used in this paper.
\begin{defi} (\cite{2}) Let $z_{0}$ be a zero of $f-a$ of multiplicity $p$ and a zero of $g-a$ of multiplicity $q$.
\begin{enumerate}
\item [i)] By $\ol N_{L}(r,a;f)$, we denote the reduced counting function of those $a$-points of $f$ and $g$ where $p>q\geq 1$,
\item [ii)]by $N^{1)}_{E}(r,a;f)$, we denote the counting function of those $a$-points of $f$ and $g$ where $p=q=1$ and
\item [iii)] by $\ol N^{(2}_{E}(r,a;f)$, we denote the reduced counting function of those $a$-points of $f$ and $g$ where $p=q\geq 2$.
\end{enumerate}
In the same way, one can define $\ol N_{L}(r,a;g),\; N^{1)}_{E}(r,a;g),\; \ol N^{(2}_{E}(r,a;g).$
\end{defi}
\begin{defi} (\cite{2}) Let $f$ and $g$ share a value $(a,0)$. Then by $\ol N_{*}(r,a;f,g)$, we denote the reduced counting function of those $a$-points of $f$ whose multiplicities differ from the multiplicities of the corresponding $a$-points of $g$.\\
Thus $\ol N_{*}(r,a;f,g)\equiv\ol N_{*}(r,a;g,f)$ and $\ol N_{*}(r,a;f,g)=\ol N_{L}(r,a;f)+\ol N_{L}(r,a;g)$.
\end{defi}
For the last few decades, the uniqueness theory of entire and meromorphic functions has been grown up as an important subfield of the value distribution theory. One of the prominent branch of the uniqueness literature is to consider the uniqueness of a meromorphic functions and its derivative sharing a small function. The inception of this particular field is due to Rubel-Yang (\cite{7}).\par
In 1977, Rubel-Yang (\cite{7}) proved that if a non-constant entire function $f$ and $f^{'}$ share two distinct finite numbers $(a,\infty)$, $(b,\infty)$, then $f = f^{'}$.\par
In 1979, further improvement in this direction was obtained by Mues-Steinmetz(\cite{6}) in this following manner:
\begin{theoA}(\cite{6}) Let $f$ be a non-constant entire function. If $f$ and $f^{'}$ share two distinct values $(a,0)$, $(b,0)$, then $f^{'}\equiv f$.
\end{theoA}
To find the specific form of the function, Yang-Zhang (\cite{7.1}) were the first to subtly investigate the value sharing of  a power of a meromorphic (resp. entire) function $F = f^{m}$ and its derivative $F'$.  Yang-Zhang (\cite{7.1}) proved the following theorem:
\begin{theoB}(\cite{7.1}) Let $f$ be a non-constant meromorphic (resp. entire) function and $m>12$ (resp. $7$) be an integer. If $F$ and $F'$ share $1$ CM, then $F = F'$ and f assumes the form $f(z) = ce^{\frac{z}{m}}$, where $c$ is a non-zero constant.
\end{theoB}
In this direction, Zhang (\cite{8}) further improved the above result in the following manner:
\begin{theoC}(\cite{8})  Let $f$ be a non-constant entire function, $m$ and $k$ be positive integers and $a(z)(\not\equiv 0,\infty)$ be a small function with respect to $f$. Suppose $f^{m} -a$ and $(f^{m})^{(k)}-a$ share the value $0$ CM and $m > k + 4$.
Then $f^m\equiv (f^m)^{(k)}$ and $f$ assumes the form $f(z) = ce^{\frac{\lambda}{m}z}$ where $c$ is a non-zero constant and $\lambda^k = 1$.
\end{theoC}
\begin{theoD}(\cite{8})  Let $f$ be a non-constant meromorphic function, $m$
 and $k$ be positive integers and $a(z)(\not\equiv 0,\infty)$ be a small function with respect to $f$. Suppose $f^{m}-a$ and $(f^{m})^{(k)}-a$ share the value $0$ CM and $(m-k-1)(m-k-4) > 3k + 6$.
Then the conclusion of {\em Theorem C} holds.
\end{theoD}
In the same year, Zhang-Yang (\cite{9}) further improved {\it Theorem D} by reducing the lower bound of $m$.
\begin{theoE}(\cite{9})  Let $f$ be a non-constant meromorphic function, $n$ and $k$ be positive integers and $a(z)(\not\equiv 0,\infty)$ be a small function with respect to $f$. Suppose $f^{n}-a$ and $(f^{n})^{(k)}-a$ share the value $0$ IM and $n > 2k+3+\sqrt{(2k+3)(k+3)}$. Then $f^n\equiv (f^n)^{(k)}$ and $f$ assumes the form $f(z) = ce^{\frac{\lambda}{n}z}$ where $c$ is a non-zero constant and $\lambda^k = 1$.
\end{theoE}
Since then, a number of improvements and generalizations have been made by many mathematicians on the uniqueness of $f^{m}$ and $(f^{m})^{(k)}$. But none of the researchers were being engaged towards changing of the sharing environment in those result. Though some recent papers (\cite{2a})-(\cite{2b}), focused on the derivative and set sharing of meromorphic functions but the expositions of those papers were different. In this paper, we take the problem in our concern, which certainly leads towards the following question :
\begin{question} If $f^{m}$ and $(f^{m})^{(k)}$ share a set $S$ instead of a value $a(\not=0,\infty)$, then can the conclusion of {\it Theorem C} be obtained? And if it is true then what is the minimum cardinality of the set $S$?
\end{question}
The following example shows that \emph{the minimum cardinality of such sets is at least three}.
\begin{exm}\label{111} Let $S=\{a,b\}$, where $a$ and $b$ are any two distinct complex numbers and $m\geq 1$ be any integer. Let $f(z)=(e^{-z}+a+b)^{\frac{1}{m}}$, where we take the principal branch when $m\geq 2$. Then $E_{f^m}(S)=E_{(f^{m})'}(S)$ but $f^{m}\not\equiv (f^{m})'$.
\end{exm}
The aim of this paper is to find the possible answer of the above question.
\section{main Results}
For a positive integer $n$, let  $P(z)$ denotes the following well known polynomial as introduced by Yi (\cite{13a}).
\be\label{e5.1}
P(z)=z^{n}+az^{n-1}+b~~~\text{where}~~ab\not=0~\text{and}~\frac{b}{a^{n}}\not=(-1)^{n}\frac{(n-1)^{(n-1)}}{n^{n}}.
\ee
Clearly, $P(z)$ has no repeated root.
\begin{theo}\label{thB1} Let $n(\geq 4)$, $k(\geq1)$ and $m(\geq k+1)$ be three positive integers. Suppose that $S=\{z :P(z)=0 \}$ where $P(z)$ is defined by (\ref{e5.1}). Let $f$ be a non-constant entire function such that $E_{f^{m}}(S,l)=E_{(f^{m})^{(k)}}(S,l)$. If
\begin{enumerate}
\item $1\leq l\leq \infty$ and $n\geq 4$, or
\item $l=0$ and  $ n\geq 5$,
\end{enumerate}
then  $f^{m} \equiv (f^{m})^{(k)}$ and $f$ takes the form $f(z)=ce^{\frac{\zeta}{m}z},$ where $c$ is a non-zero constant and $\zeta^{k}=1$.
\end{theo}
\begin{theo}\label{thB2}
Let $n(\geq 4)$, $k(\geq1)$ and $m(\geq k+1)$ be three positive integers. Suppose that $S=\{z :P(z)=0 \}$ where $P(z)$ is defined by (\ref{e5.1}). Let $f$ be a non-constant meromorphic function such that $E_{f^{m}}(S,l)=E_{(f^{m})^{(k)}}(S,l)$. If
\begin{enumerate}
\item $2\leq l\leq \infty$ and $(n-2)(2n^{2}l-5nl-3n+l+1)>6(n-1)l$, or
\item $l=1 $ and $n\geq 5,$ or
\item $l=0$ and  $ n\geq 7,$
\end{enumerate}
then  $f^{m} \equiv (f^{m})^{(k)}$ and $f$ takes the form $f(z)=ce^{\frac{\zeta}{m}z},$ where $c$ is a non-zero constant and $\zeta^{k}=1$.
\end{theo}
\begin{cor}
Let $n(\geq 4)$, $k(\geq1)$ and $m(\geq k+1)$ be three positive integers. Suppose that $S=\{z :P(z)=0 \}$ where $P(z)$ is defined by (\ref{e5.1}). Let $f$ be a non-constant meromorphic function such that $E_{f^{m}}(S,3)=E_{(f^{m})^{(k)}}(S,3)$.\par Then  $f^{m} \equiv (f^{m})^{(k)}$ and $f$ takes the form $f(z)=ce^{\frac{\zeta}{m}z},$ where $c$ is a non-zero constant and $\zeta^{k}=1$.
\end{cor}
\begin{rem} However the following questions are still open for future studies:
\begin{enumerate}
\item [i)]Is it possible to omit the condition $m\geq k+1$ keeping the cardinality of the set $S$ same ?
\item [ii)] We see that from Example \ref{111} that the cardinality of $S$ is atleast three. So for the meromorphic functions in the  Theorem \ref{thB1}, is it possible to reduce the cardinality of $S$ to three?
\end{enumerate}
\end{rem}
\section {Lemmmas}
Throughout this paper, we use the following terminologies :\\
$$T(r):=T(r,f^m)+T(r,(f^m)^{(k)})~~~\text{and}~~~S(r):=S(r,f);$$
$$F:=R(f^{m})~~~\text{and}~~~G:=R((f^{m})^{(k)})~~~\text{where}~~~R(z):=-\frac{z^{n}+az^{n-1}}{b}.$$
Also we shall denote by $H$ (\cite{7.2}) the following function
\be H:=\left(\frac{F''}{F'}-\frac{2F'}{F-1}\right)-\left(\frac{G''}{G'}-\frac{2G'}{G-1}
\right).\ee
\begin{lem}\label{l.n.1}
If $F$ and $G$ share $(1,0)$, then
\begin{enumerate}
\item [i)] $\overline{N}_{L}(r,1;F)\leq \overline{N}(r,0;f)+\overline{N}(r,\infty;f)+S(r,f)$ and
\item [ii)] $\overline{N}_{L}(r,1;G)\leq \overline{N}(r,0; (f^{m})^{k})+\overline{N}(r,\infty;f)+S(r,f).$
\end{enumerate}
\end{lem}
\begin{proof}
Clearly
\beas \overline{N}_{L}(r,1;F) &\leq&  N(r,1;F)-\ol{N}(r,1,F)\\
&\leq&  N\left(r,\infty;\frac{(f^{m})'}{f^{m}}\right)+S(r,f)\\
&\leq& (\overline{N}(r,0;f^{m})+\overline{N}(r,\infty;f^{m}))+S(r,f)\\
&\leq& (\overline{N}(r,0;f)+\overline{N}(r,\infty;f))+S(r,f).
\eeas
Similarly, we can get the inequality for $\overline{N}_{L}(r,1;G)$.
\end{proof}
\begin{lem}\label{l.n.1m}
If $F$ and $G$ share $(1,l)$ where $l\geq1$, then
$$\overline{N}_{L}(r,1;F)+\overline{N}_{L}(r,1;G) \leq \frac{1}{l}(\overline{N}(r,0;f)+\overline{N}(r,\infty;f))+S(r,f).$$
\end{lem}
\begin{proof} Clearly
\beas\overline{N}_{L}(r,1;F)+\overline{N}_{L}(r,1;G) &\leq& \overline{N}(r,1;F|\geq l+1)\\
&\leq& \frac{1}{l} (N(r,1;F)-\ol{N}(r,1,F))\\
&\leq& \frac{1}{l} N\left(r,\infty;\frac{(f^m)'}{f^{m}}\right)+S(r,f)\\
&\leq& \frac{1}{l}(\overline{N}(r,0;f)+\overline{N}(r,\infty;f))+S(r,f).
\eeas
Hence the proof.
\end{proof}
\begin{lem}\label{bc121} Assume that $F$ and $G$ share $(1,l)$, $F \not\equiv G$ and $m\geq k+1$.
\begin{enumerate}
\item [i)] If $f$ is meromorphic function, $l=0$ and $n\geq 5$, then
\beas\label{bc111} \ol{N}(r,0;f)\leq \ol{N}(r,0,(f^m)^{(k)}) &\leq& \frac{3}{n-4}\overline{N}(r,\infty;f)+S(r).\eeas
\item [ii)] If $f$ is entire function, $l=0$ and $n\geq 5$, then
$$\ol{N}(r,0;f)= \ol{N}(r,0,(f^m)^{(k)})=S(r).$$
\item [iii)] If $f$ is meromorphic function, $l\geq 1$ and $n>2+\frac{1}{l}$, then
\beas\label{bc111} \ol{N}(r,0;f)\leq \ol{N}(r,0,(f^m)^{(k)}) &\leq& \frac{l+1}{nl-2l-1}\overline{N}(r,\infty;f)+S(r).\eeas
\item [iv)] If $f$ is entire function, $l\geq 1$ and $n> 2+\frac{1}{l}$, then
$$\ol{N}(r,0;f)= \ol{N}(r,0,(f^m)^{(k)})=S(r).$$
\end{enumerate}
\end{lem}
\begin{proof} For this lemma, we define $U:=\left(\frac{F'}{(F-1)}-\frac{G'}{(G-1)}\right)$.\\
\textbf{Case-1} $U\equiv 0$.\par
On integration, we get \bea\label{arab} F-1=B(G-1).\eea
If $z_{0}$ be a zero of $f$, then equation (\ref{arab}) yields that $B=1$, which is impossible. Thus  $\ol{N}(r,0;f)=S(r,f)$.\\
\textbf{Case-2} $U\not\equiv 0$.\par
If $z_{0}$ is a zero of $f$ of order $t$, then it is zero of $F$ of order $mt(n-1)$ and that of $G$ is of order $(mt-k)(n-1)$. Hence $z_{0}$ is a zero of $U$ of order atleast $\nu=(n-2).$
Thus
\beas &&\nu \ol{N}(r,0;f)\\
&\leq& \nu \ol{N}(r,0;(f^m)^{(k)})\\
&\leq& N(r,0;U) \leq N(r,\infty;U)+S(r,f)\\
&\leq& \ol{ N}_{L}(r,1;F)+\ol{N}_{L}(r,1;G)+\ol{N}_{L}(r,\infty;F)+\ol{N}_{L}(r,\infty;G)+ S(r,f)\\
&\leq& \ol{N}_{L}(r,1;F)+\ol{N}_{L}(r,1;G)+\ol{N}(r,\infty;f)+S(r,f).
\eeas
Next we consider two subcases:\\
\textbf{Subcase-2.1} Assume $l=0$.\\
Then in view of Lemma \ref{l.n.1}, we have,
\beas && \nu\ol{N}(r,0;f)\leq \nu \ol{N}(r,0,(f^m)^{(k)}) \\
&\leq& \overline{N}(r,0;f)+\overline{N}(r,0;(f^{m})^{(k)})+3\overline{N}(r,\infty;f)+S(r,f),\\
&\leq& 2\overline{N}(r,0;(f^{m})^{(k)})+3\overline{N}(r,\infty;f)+S(r).
\eeas
Thus
\beas \ol{N}(r,0;f)\leq \ol{N}(r,0,(f^m)^{(k)}) &\leq& \frac{3}{n-4}\overline{N}(r,\infty;f)+S(r).\eeas
\textbf{Subcase-2.2} Assume $l\geq 1$.\\
In this case, instead of Lemma \ref{l.n.1}, we use Lemma \ref{l.n.1m} and obtain
\beas &&\nu\ol{N}(r,0;f)\leq \nu\ol{N}(r,0,(f^m)^{(k)})\\
&\leq& \frac{1}{l}\left(\overline{N}(r,0;f)+\overline{N}(r,\infty;f)\right)+\overline{N}(r,\infty;f)+S(r,f),\eeas
i.e.,
\beas \ol{N}(r,0;f)\leq \ol{N}(r,0,(f^m)^{(k)}) &\leq& \frac{l+1}{nl-2l-1}\overline{N}(r,\infty;f)+S(r).
\eeas
\end{proof}
\begin{lem}\label{bc123} Assume that $F$ and $G$ share $(1,l)$, $F \not\equiv G$ and $m\geq k+1$.
\begin{enumerate}
\item [i)] If $l=0$ and $n\geq 6$, then
\beas\label{el.1}
\ol{N}(r,\infty;f)\leq \frac{n-4}{2n^{2}-11n+3}T(r)+S(r).\eeas
\item [ii)] If $l\geq 1$ and $n\geq 4$, then
\beas&&  \ol{N}(r,\infty;f)\leq \frac{(nl-2l-1)l}{(2nl-l-1)(nl-2l-1)-(l+1)^{2}}T(r)+S(r).\eeas
\end{enumerate}
\end{lem}
\begin{proof} For this lemma, we define $V:=\left(\frac{F'}{F(F-1)}-\frac{G'}{G(G-1)}\right).$\\
\textbf{Case-1} $V\equiv 0$.\\
On integration, we get
\bea\label{arab1} \left(1-\frac{1}{F}\right)=A\left(1-\frac{1}{G}\right).
\eea
As $f^{m}$ and $(f^{m})^{(k)}$ share $(\infty,0)$, so if $\overline{N}(r,\infty;f)\neq S(r,f)$, then equation (\ref{arab1}) yields that $A=1$, but this is not possible. Hence $\overline{N}(r,\infty;f)=S(r,f).$\\
\textbf{Case-2} $V\not\equiv 0$.\\
Let $z_{0}$ be a pole of $f$ of order $p$, then it is a pole of $(f^{m})^{(k)}$ of order $(pm+k)$. Thus $z_{0}$ is the pole of $F$ and $G$ of order $pmn$ and $(pm+k)n$ respectively. Hence $z_{0}$ is a zero of $\left(\frac{F'}{F-1}-\frac{F'}{F}\right)$ of order atleast $(pmn-1)$ and zero of $V$ of order atleast $\lambda$ where $\lambda=2n-1$.
Thus \beas &&\lambda \overline{N}(r,\infty;f)\\
&\leq& N(r,0;V) \leq N(r,\infty;V)+S(r,f)\\
&\leq& \overline{N}_{L}(r,1;F)+\overline{N}_{L}(r,1;G)+\overline{N}_{L}(r,0;F)\\
&+& \overline{N}_{L}(r,0;G)+\ol{N}(r,0;G| F\neq 0)+S(r,f)\\
&\leq& \overline{N}_{L}(r,1;F)+\overline{N}_{L}(r,1;G)+\overline{N}(r,0;G)+S(r,f)\\
&\leq& \overline{N}_{L}(r,1;F)+\overline{N}_{L}(r,1;G)+\overline{N}(r,0;(f^{m})^{(k)})+T(r)+S(r)
\eeas
Now we consider two subcases:\\
\textbf{Case-2.1} Assume $l= 0$.\\
Then using the Lemma \ref{l.n.1}, we can get
\beas && \lambda \overline{N}(r,\infty;f) \\
&\leq& \overline{N}(r,0;f)+\overline{N}(r,0;(f^{m})^{(k)})+2\overline{N}(r,\infty;f)\\
&+&\overline{N}(r,0;(f^{m})^{(k)})+T(r)+ S(r).
\eeas
Then applying Lemma \ref{bc121}, we get
\beas (2n-3)\ol{N}(r,\infty;f) &\leq& T(r)+3\overline{N}(r,0;(f^{m})^{(k)})+S(r)\\
 &\leq&  T(r)+\frac{9}{n-4}\overline{N}(r,\infty;f)+S(r).
\eeas
Thus
\beas \ol{N}(r,\infty;f)\leq \frac{n-4}{2n^{2}-11n+3}T(r)+S(r).\eeas
\textbf{Case-2.2} Assume $l\geq 1$.\\
Here instead of Lemma \ref{l.n.1}, we use the Lemma \ref{l.n.1m} and get
\beas &&\lambda\overline{N}(r,\infty;f) \\
&\leq& \frac{1}{l}(\overline{N}(r,0;f)+\overline{N}(r,\infty;f))+\overline{N}(r,0;(f^{m})^{(k)})+T(r)+ S(r).
\eeas
That is,
\beas (l\lambda-1)\ol{N}(r,\infty;f) &\leq& (l+1)\overline{N}(r,0;(f^{m})^{(k)}+l T(r)+S(r)\\
 &\leq& l T(r)+\frac{(l+1)^{2}}{nl-2l-1}\overline{N}(r,\infty;f)+S(r).
\eeas
Thus
\beas \ol{N}(r,\infty;f)\leq \frac{(nl-2l-1)l}{(2nl-l-1)(nl-2l-1)-(l+1)^{2}}T(r)+S(r),\eeas
where $(2nl-l-1)(nl-2l-1)-(l+1)^{2}=l\big(l\{n(2n-5)+1\}-(3n-1)\big)\geq2l>0$, as $n\geq4 $ and $l\geq1$.
\end{proof}
\begin{lem}\label{bc1234}If $F$ and $G$ share $(1,l)$, $H \not\equiv0$ and $m\geq k+1$, then
\bea\label{el.2} && N(r,\infty;H)\\
\nonumber &\leq& \overline{N}(r,\infty;f)+\overline{N}\left(r,0;(f^m)^{(k)}\right)+\overline{N}\left(r,-a\frac{n-1}{n};f^m\right)\\
\nonumber &+&\overline{N}\left(r,-a\frac{n-1}{n};(f^m)^{(k)}\right)+\overline{N}_{L}(r,1;F)+\overline{N}_{L}(r,1;G)\\
\nonumber &+& \overline{N}_{0}(r,0;(f^{m})')+\overline{N}_{0}\left(r,0;(f^m)^{(k+1)}\right),
\eea
where $\overline{N}_{0}(r,0;(f^{m})')$ denotes the counting function of the zeros of $(f^{m})'$ which are not the zeros of $f^{m}(f^{m}+a\frac{n-1}{n})$ and $F-1$. Similarly, $\overline{N}_{0}\left(r,0;(f^{m})^{(k+1)}\right)$ is defined.
\end{lem}
\begin{proof} The proof is obvious if we  keep the following things in our mind:\\
Zeros of $F$ comes from zeros of $f^{m}$ and zeros of $f^{m}+a$. Also zeros of $G$ comes from zeros of $(f^{m})^{(k)}$ and  $(f^{m})^{(k)}+a$.\par
 Again, any zero of $f^{m}$ is a zero of $(f^{m})^{(k)}$ as $m\geq k+1$. Thus
$$\ol{N}_{L}(r,0;f^{m})+\ol{N}_{L}(r,0;(f^{m})^{(k)})+\ol{N}_{L}(r,0;(f^{m})^{(k)}\mid f^{m}\not= 0)\leq \ol{N}(r,0;(f^{m})^{(k)}).$$
\end{proof}
\begin{lem}\label{ma}(\cite{1a})
If $F$ and $G$ share $(1,l)$ where $0\leq l<\infty$, then
\beas&&\ol{N}(r,1;F)+\ol{N}(r,1;G)-N_{E}^{1}(r,1,F)+\left(l-\frac{1}{2}\right)\ol{N}_{*}(r,1;F,G)\\
&\leq&\frac{1}{2}(N(r,1;F)+N(r,1;G)).\eeas
\end{lem}
\begin{lem}\label{abc121} If  $f^m=(f^{m})^{(k)}$ and $m\geq k+1$, then  $f$ takes the form  $$f(z)=ce^{\frac{\zeta}{m}z},$$
where $c$ is a non-zero constant and $\zeta^{k}=1$.
\end{lem}
\begin{proof}
If $f^m=(f^{m})^{(k)}$, then we claim that $0$ and $\infty$ are the Picard exceptional value of $f$.\par
For the proof, if $z_{0}$ is a zero of $f$ of order $t$, then it is zero of $f^m$ and $(f^m)^{(k)}$ of order $mt$ and $(mt-k)$ respectively, which is impossible.\par
Similarly, if $z_{0}$ is a pole of $f$ of order $s$, then it is pole of $f^m$ and $(f^m)^{(k)}$ of order $ms$ and $(ms+k)$ respectively, which is again impossible.\par
Thus $f$ takes the form of $$f(z)=ce^{\frac{\zeta}{m}z},$$
where $c$ is a non-zero constant and $\zeta^{k}=1$.
\end{proof}
\begin{lem}\label{sr1} If $H\equiv 0$, $m\geq k+1$ and $n \geq 4$, then $F\equiv G$.
\end{lem}
\begin{proof}
It is easy to observe that $F$ and $G$ share $(1,\infty)$ and $(\infty,0)$. Now we integrate $H\equiv0$ twice and get
\bea\label{pe1.1} F\equiv \frac{AG+B}{CG+D},\eea
where $A,B,C,D$ are constant satisfying $AD-BC\neq0$.\\
Thus by Mokhon'ko's Lemma (\cite{5})
\bea\label{pe1.2} T(r,f^m)=T(r,(f^m)^{(k)})+S(r,f).\eea
Next we consider the following two cases:\\
\textbf{Case-1} $C\neq 0$. \par
Clearly $\overline{N}(r,\infty;f)=S(r,f)$, otherwise if $z_{0}$ be a pole of $f$ with multiplicity $t$, then $z_{0}$ is a pole of $F$ with multiplicity $mnt$ where as $z_{0}$ is removable singularity or analytic point of $\frac{AG+B}{CG+D}$, which is impossible as $n\geq4$ and equation (\ref{pe1.1}) holds.\par
\textbf{Subcase-1.1} $A\neq0$.\\
In this case, (\ref{pe1.1}) can be written as \bea F-\frac{A}{C}=\frac{BC-AD}{C(CG+D)}.\eea
So, $$\overline{N}\left(r,\frac{A}{C};F\right)=\overline{N}(r,\infty;G).$$
Hence using the Second Fundamental Theorem, we get
\beas&& n T(r,f^{m})+O(1)=T(r,F)\\
 &\leq& \overline{N}(r,\infty;F)+\overline{N}(r,0;F)+\overline{N}\left(r,\frac{A}{C};F\right)+S(r,F)\\
&\leq& 2\overline{N}(r,\infty;f)+\overline{N}(r,-a;f^m)+\overline{N}(r,0;f^m)+S(r,f)\\
&\leq& \overline{N}(r,-a;f^m)+\overline{N}(r,0;f^m)+S(r,f),
\eeas
which is a contradiction as $n\geq 3$.\par
\textbf{Subcase-1.2} $A=0$.\\
Then obviously $B\neq0$ and (\ref{pe1.1}) can be written as
\bea\label{arab5} F=\frac{1}{\gamma G+\delta},\eea
where $\gamma=\frac{C}{B}$ and $\delta=\frac{D}{B}$.\\
If $F$ has no 1-point, then by  the Second Fundamental Theorem, we have
\beas && n T(r,f^{m})+O(1)=T(r,F)\\
&\leq& \overline{N}(r,\infty;F)+\overline{N}(r,0;F)+\overline{N}(r,1;F)+S(r,F)\\
&\leq& \overline{N}(r,\infty;f^m)+\overline{N}(r,-a;f^m)+\overline{N}(r,0;f^m)+S(r,f)\\
&\leq& \overline{N}(r,-a;f^m)+\overline{N}(r,0;f^m)+S(r,f),
\eeas
which is a contradiction as $n\geq 3$.\\
Hence $\gamma+\delta=1$ and $\gamma\neq0$. So, equation (\ref{arab5}) becomes,
\beas F=\frac{1}{\gamma G+1-\gamma}.\eeas
Consequently, $\overline{N}\left(r,0;G+\frac{1-\gamma}{\gamma}\right)=\overline{N}(r,\infty;F)$.\\
If $\gamma\neq1$, then using Second Fundamental Theorem and equation (\ref{pe1.2}), we get
\beas && n T(r,(f^{m})^{(k)})+O(1)=T(r,G)\\
 &\leq& \overline{N}(r,\infty;G)+\overline{N}(r,0;G)+\overline{N}\left(r,0;G+\frac{1-\gamma}{\gamma}\right)+S(r,G)\\
&\leq& 2\overline{N}(r,\infty;(f^m)^{(k)})+\overline{N}(r,-a;(f^m)^{(k)})+\overline{N}(r,0;(f^m)^{(k)})+S(r)\\
&\leq& \overline{N}(r,-a;(f^m)^{(k)})+\overline{N}(r,0;(f^m)^{(k)})+S(r),\eeas
which is a contradiction as $n\geq 3$.\\
Hence $\gamma=1$, i.e., $FG\equiv 1$, i.e.,
$$\big(f^{m}\big)^{n-1}\big(f^{m}+a\big)\bigg((f^{m})^{(k)}\bigg)^{n-1}\bigg((f^{m})^{(k)}+a\bigg)\equiv b^{2}.$$
As $f^{m}$ and $(f^{m})^{(k)}$ share $(\infty,0)$, so $\infty$ is the exceptional values of $f^{m}$. Also as zeros of $F$ is neutralized by poles of G, so $0,-a$ are also exceptional values of $f^{m}$. But this is not possible by Second Fundamental Theorem. Thus $FG \not\equiv 1$.\\
\textbf{Case-2} $C=0$.\\
Then $AD\neq0$ and equation (\ref{pe1.1}) can be written as
$$F=\lambda G+\mu,$$
where $\lambda=\frac{A}{D}$ and $\mu=\frac{B}{D}$. Also $F$ and $G$ share $(\infty,\infty)$.\par
Hence in this case, we have
\begin{equation}\label{hindi}
\ol{N}(r,\infty;f)=S(r,f).\end{equation}
Because otherwise, if  $z_{0}$ be a pole of $f$ of order $t$, then it is a pole of $F$ of order $mtn$ and that of $G$ is $(mt+k)n$. But $F$ and $G$ share poles counting multiplicities. Thus $mtn=(mt+k)n$, but $nk\not= 0$ by assumption on $n,k$.\par
Also, proceeding as the previous case, we can conclude that $F$ has atleast one $1$-point. Thus $\lambda+\mu=1$ with $\lambda\neq0$.\par
\textbf{Subcase-2.1} Assume $\lambda\not=1$.\\
Then $\ol{N}(r,0;f)=S(r,f),$ because otherwise, if $z_{0}$ is a zero of $f$, then it is zero of $F$ as well as $G$ as $m\geq k+1$, consequently, $\mu=0$, which is a contradiction.\\
Now using the Second Fundamental Theorem, equations (\ref{pe1.2}), (\ref{hindi}), we get
\beas && n T(r,f^{m})+O(1)=T(r,F)\\
&\leq& \overline{N}(r,\infty;F)+\overline{N}(r,0;F)+\overline{N}(r,1-\lambda;F)+S(r,F)\\
&\leq& \overline{N}(r,0;f^m)+\overline{N}(r,-a;f^m)+\overline{N}\big(r,0;(f^m)^{(k)}\big)+\overline{N}\big(r,-a;(f^m)^{(k)}\big)\\
&+& S(r,f)\\
&\leq& 3  T(r,f^{m})+S(r,f),\eeas
which is a contradiction as $n\geq 4$.\par
\textbf{Subcase-2.2}  $\lambda=1$.\\
Consequently, we have $F\equiv G$.
\end{proof}

\begin{lem}\label{sr2} If $F\equiv G$, then $f^{m}\equiv (f^{m})^{(k)}$ for any $m\geq 1$, $k \geq 1$ and $n \geq 4$.
\end{lem}
\begin{proof}
We see that $F$ and $G$ share $(1,\infty)$ and $(\infty,\infty)$; $\ol{N}(r,\infty;f)=S(r,f).$\par Let $\lambda_{i}$ $(i=1,2,...,n-1)$ be the non-real distinct zeros of $h^{n}-1=0$.\par
Next we put $h:=\frac{(f^{m})^{(k)}}{f^{m}}$ in $F\equiv G$ and get $f^{m}=-a\frac{h^{n-1}-1}{h^{n}-1}$.\par
If $h$ is non-constant meromorphic function, then by Second Fundamental Theorem and Mokhon'ko's Lemma, we have
\beas &&(n-3)T(r,h)\\
&\leq& \sum\limits_{i=1}^{n-1}\ol{N}(r,\lambda_{i};h)+S(r,h)\\
&\leq& \ol{N}(r,\infty;f)+S(r,f)\\
&\leq& S(r,f),
\eeas
which is a contradiction as $n\geq 4$.\\
Thus $h$ is constant. Consequently the only possibility of $h$ is 1 and hence $f^{m}\equiv (f^{m})^{(k)}$.
\end{proof}
\section {Proof of the theorems}
\begin{proof} [\textbf{Proof of Theorem \ref{thB1} and \ref{thB2}}] We consider two cases:\\
\textbf{Case-1} Assume $H \not\equiv 0$. Then $F \not\equiv G$ and $$\overline{N}(r,1;F|=1)=\overline{N}(r,1;G|=1)\leq N(r,\infty;H).$$
Now using the Second Fundamental Theorem and  Lemma \ref{bc1234}, we get
\bea\label{pe1.1.1}&& (n+1) \{T(r,f^m)+T(r,(f^m)^{(k)})\}\\
\nonumber &\leq& \overline{N}(r,\infty;f^{m})+\overline{N}(r,\infty;(f^{m})^{(k)})+\overline{N}(r,0;f^{m})+\overline{N}(r,0;(f^{m})^{(k)})\\
\nonumber &+& \overline{N}\left(r,-a\frac{n-1}{n};f^m\right)+\overline{N}\left(r,-a\frac{n-1}{n};(f^m)^{(k)}\right)+\overline{N}(r,1;F)\\
\nonumber &+& \overline{N}(r,1;G)- N_{0}(r,0,(f^m)')-N_{0}\left(r,0,(f^m)^{(k+1)}\right)+S(r)\\
\nonumber &\leq& 3\{\overline{N}(r,\infty;f)+\overline{N}(r,0;(f^m)^{(k)})\}+\overline{N}(r,1;F)+\overline{N}(r,1;G)\\
\nonumber &+& 2T(r) -\overline{N}(r,1;F|=1)+\overline{N}_{L}(r,1;F)+\overline{N}_{L}(r,1;G)+S(r).
\eea
Thus in view of Lemma \ref{ma}, we get
\bea\label{pe1.1.1} \left(\frac{n}{2}-1\right)T(r) &\leq& 3\{\overline{N}(r,\infty;f)+\overline{N}(r,0;(f^m)^{(k)})\}\\
\nonumber &+& \left(\frac{3}{2}-l\right)\{\overline{N}_{L}(r,1;F)+\overline{N}_{L}(r,1;G)\}+S(r).
\eea
\textbf{Subcase-1.1} $l\geq2$.\\
Now using Lemmas \ref{bc121} and \ref{bc123} in (\ref{pe1.1.1}), we get
\bea\label{pe1.1.1ma}&& \left(\frac{n}{2}-1\right)T(r)\leq \frac{3(n-1)l}{(n-2)l-1}\overline{N}(r,\infty;f)+S(r).
\eea
That is,
\bea\label{pe1.1.1mab}&& (\frac{n}{2}-1)T(r)\\
\nonumber &\leq& \frac{3(n-1)l^{2}}{(2nl-l-1)(nl-2l-1)-(l+1)^{2}}T(r)+S(r)\\
\nonumber &\leq& \frac{3(n-1)l}{2n^{2}l-5nl-3n+l+1}T(r)+S(r).
\eea
Thus if $n\geq 4$ \big( resp. $(n-2)(2n^{2}l-5nl-3n+l+1)>6(n-1)l$\big) and  $f$ is entire (resp. meromorphic) function, then the inequality (\ref{pe1.1.1ma}) (resp. \ref{pe1.1.1mab})  leads to a contradiction.\\
\textbf{Subcase-1.2} $l=1$.\\
Applying Lemmas \ref{l.n.1m}, \ref{bc121} and \ref{bc123} in (\ref{pe1.1.1}), we get
\bea\nonumber \left(\frac{n}{2}-1\right)T(r) &\leq& 3\{\overline{N}(r,\infty;f)+\overline{N}(r,0;(f^m)^{(k)})\}\\
\nonumber &+& \frac{1}{2}\{\overline{N}(r,0;f)+\overline{N}(r,\infty;f)\}+S(r)\\
\label{pe1.1.1maa} &\leq& \frac{7(n-1)}{2(n-3)}\overline{N}(r,\infty;f)+S(r)\\
\label{pe1.1.1maab}&\leq& \frac{7(n-1)}{2\{(2n-2)(n-3)-4\}}T(r)+S(r).
\eea
Thus the inequality (\ref{pe1.1.1maa}) ( resp. \ref{pe1.1.1maab}) leads to a contradiction if $f$ is entire (resp. meromorphic) function and $n\geq 4$ (resp. $5$).\\
\textbf{Subcase-1.3} $l=0$.\\
Using Lemmas \ref{l.n.1}, \ref{bc121} and \ref{bc123} in (\ref{pe1.1.1}), we get
\bea\nonumber \left(\frac{n}{2}-1\right)T(r) &\leq& 3\{\overline{N}(r,\infty;f)+\overline{N}(r,0;(f^m)^{(k)})\}\\
\nonumber &+& 3\{\overline{N}(r,0;(f^{m})^{(k)})+\overline{N}(r,\infty;f)\}+S(r)\\
\label{pe11.1.1maa} &\leq& \frac{6(n-1)}{(n-4)}\overline{N}(r,\infty;f)+S(r)\\
\label{pe11.1.1maab} &\leq& \frac{6(n-1)}{2n^{2}-11n+3}T(r)+S(r).
\eea
Thus the inequality (\ref{pe11.1.1maa}) (resp. \ref{pe11.1.1maab}) leads to a contradiction if $f$ is entire (resp. meromorphic) function and $n\geq 5$ (resp. $7$).\\
\textbf{Case-2} Next we assume that $H\equiv 0$. Then by Lemmas \ref{sr1} and \ref{sr2}, we have $f^m=(f^{m})^{(k)}$. Now by applying Lemma \ref{abc121}, we see that $f$ takes the form  $$f(z)=ce^{\frac{\zeta}{m}z},$$
where $c$ is a non-zero constant and $\zeta^{k}=1$. This completes the proof.
\end{proof}
\section{Acknowledgement}
The first author's research work is supported by the Council Of  Scientific and Industrial Research, Extramural Research Division, CSIR Complex, Pusa, New Delhi-110012, India, under the sanction project no. 25(0229)/14/EMR-II. \par
The second author's research work is supported by the Department of Science and Technology, Govt. of India under the sanction order DST/INSPIRE Fellowship/2014/IF140903.

\vspace{0.2 in}

\noindent  ABHIJIT BANERJEE,\\
Department of Mathematics, University of Kalyani, India 741235.\\
e-mail address: {\it {abanerjee\_kal@yahoo.co.in, abanerjeekal@gmail.com}}
\vspace{0.1 in}\\
\noindent  BIKASH CHAKRABORTY,\\
Department of Mathematics, University of Kalyani, India 741235.\\
Present Address :\\
Department of Mathematics, Ramakrishna Mission Vivekananda\\ Centenary College, West Bengal 700118, India.\\
e-mail address: {\it {bikashchakraborty.math@yahoo.com} }

\begin{thebibliography}{99}
\bibitem{1a} A. Banerjee, \emph{Uniqueness of meromorphic functions sharing two sets with finite weight II}, Tamkang J. Math., 41(4) (2010), 379-392.
\bibitem{2} A. Banerjee and B. Chakraborty, \emph{A new type of unique range set with deficient values},  Afr. Mat., 26(7-8) (2015), 1561-1572.
\bibitem{2a} A. Banerjee and B. Chakraborty, \emph{Uniqueness of the power of a meromorphic functions with its differential polynomial
sharing a set}, Math. Morav., 20(2) (2016), 1-14.
\bibitem{2b} A. Banerjee and B. Chakraborty, \emph{Further results on the uniqueness of meromorphic functions and their derivative counterpart sharing one or two sets}, Jordan J. Math. Stat., 9(2) (2016), 117-139.
\bibitem{3} W. K. Hayman,\emph{ Meromorphic Functions}, The Clarendon Press, Oxford (1964).
\bibitem{4} I. Lahiri, \emph{Weighted sharing and uniqueness of meromorphic functions}, Nagoya Math. J., 161 (2001), 193-206.
\bibitem{5} A. Z. Mokhon'ko, \emph{On the Nevanlinna characteristics of some meromorphic functions}, in \enquote{Theory of Functions, functional analysis and their applications}, Izd-vo Khar'kovsk, Un-ta, 14 (1971), 83-87.
\bibitem{6} E. Mues and N. Steinmetz, \emph{Meromorphe Funktionen die unit ihrer Ableitung Werte teilen}, Manuscripta Math., 29 (1979) 195-206.
\bibitem{7} L. A. Rubel and C. C. Yang,\emph{ Values shared by an entire function and its derivative}, Complex analysis (Proc. Conf., Univ. Kentucky, Lexington, Ky., 1976), Lecture Notes in Math., 599 (1977), 101-103, Springer, Berlin.
\bibitem{7.1} L. Z. Yang and J. L. Zhang, \emph{Non-existance of meromorphic solutions of Fermat type functional equation}, Aequations Math., 76 (2008), 140-150.
\bibitem{7.2} H. X. Yi and W. C. Lin, \emph{Uniqueness theorems concerning a question of Gross}, Proc. Japan Acad., 80, Ser.A (2004), 136-140.
\bibitem{8}J. L. Zhang, \emph{Meromorphic functions sharing a small function with their derivatives}, Kyungpook Math. J., 49 (2009), 143-154.
\bibitem{9} J. L. Zhang and L. Z. Yang, \emph{A power of a meromorphic function sharing a small function with its derivative}, Ann. Acad. Sci. Fenn. Math., 34 (2009), 249-260.
\bibitem{13a} H. X. Yi, \emph{On a question of Gross concerning uniqueness of entire functions}, Bull. Austral. Math. Soc., 57 (1998), 343-349.
\end{thebibliography}
\end{document}